\documentclass{amsart}

\usepackage{amsmath}
\usepackage{amsfonts}
\usepackage{amsopn}
\usepackage{amssymb}

\newcommand{\R}{\ensuremath{\mathbb{R}}}
\newcommand{\Z}{\ensuremath{\mathbb{Z}}}

\newcommand{\izero}{I}

\DeclareMathOperator{\diam}{diam}
\DeclareMathOperator{\Lip}{Lip}
\DeclareMathOperator{\id}{id}

\DeclareMathOperator{\mass}{mass}
\DeclareMathOperator{\vol}{vol}
\DeclareMathOperator{\size}{size}
\DeclareMathOperator{\FV}{FV}
\DeclareMathOperator{\supp}{supp}
\DeclareMathOperator{\interior}{int}

\newtheorem{thm}{Theorem}
\newtheorem{conj}{Conjecture}
\newtheorem{lemma}[thm]{Lemma}
\newtheorem{prop}[thm]{Proposition}
\newtheorem{cor}[thm]{Corollary}
\theoremstyle{remark}
\newtheorem*{remark}{Remark}

\title{High-dimensional fillings in Heisenberg groups}
\author{Robert Young}
\address{University of Toronto\\
  Dept.\ of Mathematics\\
  40 St. George Street, Rm.\ BA6290\\
  Toronto, ON  M5S 2E4\\
  Canada } \date{\today}
\email{ryoung@math.toronto.edu}
\begin{document}
\maketitle
\begin{abstract}
  We use intersections with horizontal manifolds to show that
  high-dimensional cycles in the Heisenberg group can be approximated
  efficiently by simplicial cycles.  This lets us calculate all of the
  higher-order Dehn functions of the Heisenberg groups, thus proving a
  conjecture of Gromov.
\end{abstract}


\bibliographystyle{plain}
\section{Introduction}
Nilpotent groups, especially Carnot groups, have scaling properties
that make them particularly useful in geometry.  They appear, for
instance, as tangent cones of sub-riemannian manifolds and in the
horospheres of negatively-curved symmetric spaces.  Their scaling
automorphisms have also been used to solve problems on fillings and
extensions of Lipschitz maps to Carnot groups.  Gromov, for instance,
used scaling techniques and microflexibility to construct extensions
from Lipschitz maps of spheres to Lipschitz maps of discs and bound
the Dehn functions of some nilpotent groups \cite{GroCC}.  In
\cite{YoungFING}, we found techniques for avoiding the use of
microflexibility and generalizing Gromov's Dehn function techniques to
higher-order Dehn functions, and Stefan Wenger and the author used
these techniques to give an elementary proof of Gromov's Lipschitz
extension theorem \cite{YW}.

We can measure the difficulty of extension problems by studying
filling invariants like the higher-order Dehn functions,
$\FV^d$.  In \cite{GroAII}, Gromov made the following conjecture:
\begin{conj}\label{conj:heisenDehn}
  In the $(2n+1)$-dimensional Heisenberg group $H_{n}$,
  \begin{align*}
    \FV^{d+1}(V)\sim V^{\frac{d+1}{d}} & \text{\quad for }1\le d<n\\
    \FV^{d+1}(V)\sim V^{\frac{d+2}{d}}& \text{\quad for }d=n,\\
    \FV^{d+1}(V)\sim V^{\frac{d+2}{d+1}}& \text{\quad for }n<d<2n+1.
  \end{align*}
\end{conj} 

In \cite{YoungFING}, we proved the first two bounds.  The main idea of
the proof was to construct a sequence $\{P_i(\alpha)\}$ of
approximations of a Lipschitz cycle $\alpha$.  Each $P_i(\alpha)$ was
an approximation of $\alpha$ by simplices of diameter $\sim 2^i$, and
for each $i$, we constructed a chain $R_i(\alpha)$ such that
$$\partial R_i(\alpha)=P_{i+1}(\alpha)-P_{i}(\alpha).$$
The sum of the $R_i$'s is then a filling of (an approximation of) the
original cycle, and bounds on the mass of the $R_i$'s lead to bounds
on the filling volume of $\alpha$.

We constructed the $P_i$'s and $R_i$'s using horizontal triangulations
and scaling automorphisms.  If $G$ is a Carnot group which admits a
triangulation whose low-dimensional simplices are horizontal (i.e.,
tangent to a certain left-invariant distribution of tangent planes),
we can use the Federer-Fleming deformation theorem \cite{FedFlem} to
approximate $\alpha$ by a simplicial cycle in that triangulation.
Rescaling the triangulation produces a family of approximations of
$\alpha$, and since the simplices in the triangulation are horizontal,
we can bound how quickly their scalings grow, leading to euclidean
bounds on the filling functions of $G$.

If $G=H_{n}$, there are many horizontal maps of $k$-simplices into
$G$ for $k\le n$, and we can use these maps to construct horizontal
triangulations, providing the first two bounds in
Conj.~\ref{conj:heisenDehn}.  On the other hand, for $k>n$, there are
no horizontal triangulations, so using the same approximation methods
as before gives poor results; if $\dim \alpha>n$, then $P_i(\alpha)$
may have much larger mass than $\alpha$, leading to poor bounds on the
high-dimensional filling functions of $G$.  In this paper, we will
provide efficient methods for approximating high-dimensional cycles in
$H_{n}$ and use them to bound the filling functions of $H_{n}$.  We will show:

\begin{thm}\label{thm:DehnThm}
  In the $(2n+1)$-dimensional Heisenberg group $H_{n}$,
  $$\FV^{d+1}(V)\sim V^{\frac{d+2}{d+1}} \text{\quad for }n<d<2n+1.$$
\end{thm}

The basic idea of our construction is to approximate a cycle $\alpha$
by counting its intersections with a dual skeleton of a triangulation.
If $\alpha$ has dimension $k>n$, cells in the approximation come from
intersections between $\alpha$ and $(2n+1-k)$-dimensional cells of the
dual skeleton.  Since $2n+1-k\le n$, we can make these cells
horizontal, thus limiting the number of possible intersections.  The
result is an approximation $P_\tau(\alpha)$ such that $\mass
\alpha\sim \mass P_\tau(\alpha)$, which can be used in constructing a
filling of $\alpha$.

In Section~\ref{sec:prelims}, we will give some background necessary
in the rest of the paper; in Section~\ref{sec:approx}, we will
construct the approximations we will need; and in
Section~\ref{sec:mainthm}, we will use these approximations to
construct fillings and bound the filling invariants of $H_{n}$.

Parts of this paper were written at the Institut des Hautes {\'E}tudes
Scientifiques and at New York University.  The author was partially
supported by a Discovery Grant from the Natural Sciences and
Engineering Research Council of Canada.

\section{Preliminaries}\label{sec:prelims}
\subsection{Filling functions}
The basic idea of a filling function is to measure the difficulty of
filling a boundary of a given size.  There are several ways to make
this rigorous, depending on the type of boundary and the type of
filling.  We will use a definition based on Lipschitz chains, as
in \cite{GroFRM}.  Definitions of other higher-order
filling invariants can be found in \cite{AlWaPr, ECHLPT, GroftA,
  GroftB}.  An integral Lipschitz $d$-chain in a space $X$
is a finite linear combination, with integer coefficients, of
Lipschitz maps from the euclidean $d$-dimensional simplex $\Delta^d$
to $X$.  We will often call this simply a Lipschitz $d$-chain.

We will generally take $X$ to be a riemannian manifold or simplicial
complex.  In this case, Rademacher's Theorem implies that a Lipschitz
map to $X$ is differentiable almost everywhere, and one can define the
volume of a Lipschitz map as the integral of the magnitude of its
jacobian.  If $a$ is a Lipschitz $d$-chain and $a=\sum_i x_i \alpha_i$
for some maps $\alpha_i:\Delta^d\to X$ and some coefficients $x_i\in
\Z, x_i\ne 0$, we define
$$\|a\|_1:=\sum_i |x_i|.$$
When $d>0$, we define the {\em mass} of $a$ to be
$$\mass{a}:=\sum_i |x_i| \vol_d \alpha_i,$$
and if $d=0$, we let $\mass{a}:=\|a\|_1$.  
We define the {\em support} of $a$ to be
$$\supp a:=\bigcup_i \alpha_i(\Delta^d).$$
The Lipschitz chains form a chain complex, which we denote
$C^{\text{lip}}_*(X)$.  If $g:X\to Y$ is a Lipschitz map, we can
define $g_\sharp$, the pushforward map, to be the linear map which
sends the simplex $\alpha:\Delta^d\to X$ to the simplex $g\circ \alpha$; this is
a map of chain complexes.  Likewise, we can define a smooth chain to be a
sum of smooth maps of simplices and denote the complex of smooth
chains by $C^{\text{sm}}_*(X)$.

If $X$ is a $d$-connected riemannian manifold and $a$ is
an integral Lipschitz $d$-cycle in $X$, we define the filling volume
of $a$ to be:
$$\FV^{d+1}_X(a):=\inf_{\partial b=a}
\mass{b},$$ where the infimum is taken over the set of $b\in
C^{\text{lip}}_{d+1}(X)$ such that $\partial b=a$.  We get the
$(d+1)$-dimensional filling volume function by taking a supremum over
all cycles of a given volume:
$$\FV^{d+1}_X(V):=\sup_{\mass a \le V} \FV^{d+1}_X(a),$$
where the supremum is taken over integral Lipschitz $d$-cycles.  

A related definition is the $d$-th order Dehn function, $\delta^d$,
which measures the volume necessary to extend a map $S^d\to X$ to a
map $D^{d+1}\to X$ (the fact that $\delta^d$ corresponds most closely
to $\FV^{d+1}$ is unfortunate but conventional).  If $X$ is a
$d$-connected manifold or simplicial complex and $f:S^d\to X$ is a Lipschitz map, we define
$$\delta^d_X(f):=\mathop{\inf_{g:D^{d+1}\to X}}_{g|_{S^d}=f} \vol_d{g}$$
and
$$\delta^d_X(V):=\mathop{\sup_{f:S^{d}\to X}}_{\vol_d f\le V} \delta^d_X(f).$$

We have defined $\delta^d$ and $\FV^{d+1}$ as invariants of spaces,
but they can in fact be defined as invariants of groups.  If $G$ is a
group which acts geometrically (that is,
cocompactly, properly discontinuously, and by isometries) on a
$d$-connected manifold or simplicial complex $X$, then the asymptotic
growth rates of $\delta^d_X(V)$ and $\FV^{d+1}_X(V)$ are invariants of
$G$.  To make this rigorous, we can define a partial ordering
on functions $\R^+\to \R^+$ so that $f\precsim g$ if and only if there
is a $c$ such that
$$f(x)\le c g(c x + c) + cx + c.$$
We say $f\sim g$ if and only if $f\precsim g$ and $g \precsim f$.
Then if
$X_1$ and $X_2$ are two $d$-connected manifolds or simplicial complexes on which $G$ acts geometrically, then
$$\delta^d_{X_1}(V)\sim \delta^d_{X_2}(V)$$
and
$$\FV^d_{X_1}(V)\sim \FV^d_{X_2}(V).$$
This is proved for a simplicial version of $\delta^d$ in
\cite{AlWaPr}, but the proof there also applies to a simplicial
version of $\FV^d$; one can show that the Lipschitz versions used here
are equivalent to simplicial versions using the Deformation
Theorem.

The relationship between $\delta^d$ and $\FV^{d+1}$ depends on $d$.
When $d\ge 3$, then $\delta^d_X\sim \FV^{d+1}_X$ for all $d$-connected
manifolds or simplicial complexes $X$.  When $d=2$, then
$\delta^d_X\lesssim \FV^{d+1}_X$.  (See \cite[App. 2.(A')]{GroFRM} for
the upper bound on $\delta^d$ and \cite[Rem. 2.6.(4)]{BBFS} for the
lower bound; see also \cite{GroftA, GroftB}.)  Since this paper focuses
on upper bounds on $\FV^{d+1}$ when $d\ge 2$, the bounds in this paper
will also hold for $\delta^d$.

\subsection{Triangulations}
A {\em triangulation} of a space $X$ consists of a simplicial complex
$\tau$ and a homeomorphism $\phi:\tau\to X$.  We can put a metric and
a smooth structure on $\tau$ so that each simplex is isometric to the
standard euclidean simplex; we say that a map $\tau\to X$ is {\em
  piecewise smooth} if it is smooth on each simplex.  When the meaning
is clear, we will sometimes refer to the triangulation $(\tau,\phi)$
as $\tau$.  We will require throughout this paper that $\phi$ is
locally Lipschitz.  If $S\subset X$ is such that $\phi^{-1}(S)$ is a
subcomplex $\tau_0$ of $\tau$, then we say that the triangulation
$(\tau,\phi)$ \emph{restricts to} $(\tau_0,\phi|_{\tau_0})$ on $S$.

If $(\tau,\phi)$ is a triangulation of $X$, then we can define the
simplicial chains of $X$ to be the push-forwards of simplicial chains
of $\tau$.  If $\phi:\tau\to X$ is piecewise smooth or Lipschitz, then
the simplicial chains of $X$ are a subcomplex of the smooth or
Lipschitz chains.  By abuse of notation, we will denote the chain
complex of simplicial chains of $X$ by $C_*(\tau)$.

If $G$ acts on $X$, then $G$ acts on the set of triangulations of $X$.
Let $\rho_g:X\to X$ be the map $\rho_g(x)=gx$.  If $(\tau,
\phi:\tau\to G)$ is a triangulation of $G$ and $g\in X$, we let
$g\cdot \tau$ be the triangulation $(\tau, \rho_g\cdot \phi)$.  Furthermore, if $\Gamma$ is a lattice in $G$, we say that
$\tau$ is {\em $\Gamma$-adapted} if $g\cdot \tau=\tau$ for every $g\in
\Gamma$.  That is, $\Gamma$ acts on $\tau$ and the map $\phi$ is
equivariant with respect to this action.  If $\tau$ is
$\Gamma$-adapted and $\sigma$ is a simplex of $\tau$, we will write
$\Gamma\cdot\sigma$ to denote the corresponding $\Gamma$-orbit of
simplices.  We will write $\Gamma \backslash \tau$ to denote the set
of orbits of simplices.

\subsection{Carnot groups}
This paper focuses on Carnot groups, which are nilpotent Lie groups
provided with a family of scaling automorphisms.  Recall that if $G$
is a simply-connected nilpotent Lie group and $\mathfrak{g}$ is its
Lie algebra, then the lower central series
$$\mathfrak{g}=\mathfrak{g}_0\supset \dots \supset
\mathfrak{g}_{k-1}=\{0\},$$
$$\mathfrak{g}_{i+1}=[\mathfrak{g}_i,\mathfrak{g}]$$
terminates.  If $\mathfrak{g}_k=\{0\}$ and $\mathfrak{g}_{k-1}\ne \{0\}$,
we say that $\mathfrak{g}$ has nilpotency class $k$.  If there is a
decomposition
$$\mathfrak{g}=V_1\oplus \dots \oplus V_k$$
such that 
$$\mathfrak{g}_i=V_{i+1}\oplus \dots \oplus V_k$$
and $[V_i,V_j]\subset V_{i+j}$ for all $i,j\le k$, we call it a
grading of $\mathfrak{g}$.  If $\mathfrak{g}$ has a grading, we can extend the
$V_i$ to left-invariant distributions on $G$ and give $G$ a
left-invariant metric such that the $V_i$'s are orthogonal.  With this
metric, $G$ is called a Carnot group.

If $G$ is a Carnot group, there is a family of automorphisms $s_t:G\to
G$ which act on the Lie algebra by $s_t(v)=t^i v$ for all $v\in V_i$.
These automorphisms distort vectors in $\mathfrak{g}$ by differing
amounts.  Vectors in $V_1$ are distorted the least, and we call these
vectors {\em horizontal}.  If $M$ is a manifold and $f:M\to
G$ is a piecewise smooth map, we say that $f$ is {\em horizontal} if
the tanget planes to $f$ all lie in the distribution $V_1$.  

\section{Approximating cycles in Carnot groups}\label{sec:approx}

It is a theorem of Federer and Fleming \cite[??]{FedFlem}
that any Lipschitz current in $\R^n$ can be approximated by a sum of
cubes in a grid in $\R^n$ such that the mass of the sum of cubes is no
more than a constant factor larger than the mass of the original
cycle; furthermore, this constant factor is independent of the
side-length of the cubes involved.  A consequence of this
\cite[??]{FedFlem} is that the filling functions of $\R^n$ satisfy
$$\FV^k_{\R^n}(V)\precsim V^{\frac{k}{k-1}}.$$

In \cite{YoungFING}, we generalized these results to chains of low
dimension in Carnot groups with horizontal triangulations (for
example, the Heisenberg groups).  In this section, we will develop
tools for constructing approximations of chains of low
\emph{co}dimension in such groups.  

If $a$ is a smooth chain in $G$ and $\tau$ is a
$\Gamma$-adapted triangulation of $G$, we can construct a dual skeleton
$\tau^*$ of $\tau$.  If $\sigma$ is a cell of $\tau$ and $\sigma^*$ is
its dual, we let $[\sigma]$ and $[\sigma^*]$ be the corresponding
simplicial chains.  For all
$g\in G$ except a measure-zero subset, we can define the intersection
number $i(a,g\cdot [\sigma^*])$ between a translation of $\sigma*$
and $a$.  Therefore, for all but a measure-zero subset of $G$, we
can define an approximation 
$$P_{g\cdot \tau}(a):=\sum_{\sigma\in \tau^{(d)}}
i(a,g\cdot [\sigma^*]) g\cdot [\sigma] \in C_{d+1}(g\cdot
\tau).$$ For any given $g$, the number of intersections between
$a$ and $g\cdot \sigma^*$ might be much larger than the mass of
$a$, but the expected number of intersections is on the order of
the mass of $a$.

For the applications in the next section, we will need a slightly more
complicated definition.  If $f:\tau\to G$ is a $\Gamma$-equivariant
Lipschitz map (not necessarily a homeomorphism) and $\sigma$ is a cell
of $\tau$, we think of the images $f(\sigma)$ and $f(\sigma^*)$ as
cells in a ``folded triangulation'' of $G$.  We can use the
methods above to approximate $a$ as a sum of ``folded cells''
$$P_{g\cdot f(\tau)}(a):=\sum_{\sigma\in \tau^{(d)}}
i(a,g\cdot f_\sharp([\sigma^*])) g\cdot f_\sharp(\sigma).$$
As with the previous construction, we can bound the expected mass of
$P_{g\cdot f(\tau)}(a)$ in terms of $f$ and $a$.  

Define the $d$-size of $f$ to be the total mass of the images of the $d$-cells of
$\sigma^*$:
$$\size_d(f)=\sum_{\Gamma\cdot{\sigma}\in  \Gamma \backslash B(\tau)^{(d)}} \mass
  f_\sharp(\sigma^*).$$
In this section, we will expand on the above constructions and prove
the following properties:
\begin{prop}\label{prop:PandQ}
  Let $0\le d\le n$.  Let $Z\subset G$ be a compact fundamental domain
  for the right action of $\Gamma$ on $G$, so that $Z\Gamma=G$.  If
  $a$ is a smooth $d$-chain in $G$, then for all but a
  measure-zero set of $g\in G$, there is a $d$-chain
  $$P_{g\cdot f(\tau)}(a)\in g\cdot f_\sharp (C_d(\tau))$$
  such that
  $$\partial P_{g\cdot f(\tau)}(a)=P_{g\cdot f(\tau)}(\partial a).$$
  There is a $c_P>0$ depending only on $d$ and $G$ such that
  \begin{equation}
    \int_Z \|P_{g\cdot f(\tau)}(a)\|_1 \; dg \le c_P
    \mass{a} \size_d(f).\label{eq:Pbound}
  \end{equation}

  If $a$ is a cycle, there is a chain $Q_{g\cdot
    f(\tau)}(a)\in C_{d+1}^\text{lip}$ such that
  $$\partial Q_{g\cdot f(\tau)}(a)=a-P_{g\cdot
    f(\tau)}(a).$$
  If $a\in C_d(\tau)$, there is a $c_Q>0$ depending on $f$,
  $\tau$, $d$, and $G$ such that
  \begin{equation}
    \int_Z \mass Q_{g\cdot f(\tau)}(a) \; dg\le  c_Q \mass a.\label{eq:Qbound}
  \end{equation}
\end{prop}

To construct $P$ and $Q$, we first define the dual complex $\tau^*$.
We will identify the cells of $\tau^*$ with subcomplexes of the
barycentric subdivision $B(\tau)$ of $\tau$.  Recall that $d$-cells of
$B(\tau)$ can be identified with flags $\sigma_0\subset \dots \subset
\sigma_d$ where the $\sigma_i$ are simplices of $\tau$.  If $\sigma$
is a $d$-simplex of $\tau$, let $\sigma^*$ be the subcomplex of
$B(\tau)$ consisting of flags $\sigma_0\subset \dots \subset
\sigma_{d'}$ where $\sigma \subset \sigma_i$ for all $i$.  This is
homeomorphic to a disc of dimension $n-d$, and these discs partition
$B(\tau)$ into a CW-complex isomorphic to $\tau^*$, so we can view
$\tau$ and $\tau^*$ as two different ways of decomposing $B(\tau)$.

If $\alpha:\tau_1\to G$ and $\beta:\tau_2\to G$ are piecewise smooth
maps from two simplicial complexes, we say that they are {\em
  transverse} if for all simplices $\Delta_1\in \tau_1$ and
$\Delta_2\in \tau_2$, the restrictions $\alpha|_{\interior \Delta_1}$
and $\beta|_{\interior \Delta_2}$ are transverse.  (In particular, if
the domains of $\alpha$ and $\beta$ are simplices, then the
restrictions of $\alpha$ and $\beta$ to any face are transverse.)  If
$a$ and $b$ are smooth chains, we say they are transverse if $a=\sum
x_i \alpha_i$ and $b=\sum y_j \beta_j$ for some maps $\alpha_i$ and
$\beta_i$ such that $\alpha_i$ is transverse to $\beta_j$ for all $i$
and $j$.

If $\alpha:\Delta^d\to G$ and
$\beta:\Delta^{d'}\to G$ are transverse, where $\Delta^d$ and
$\Delta^{d'}$ are the standard simplices of dimension $d$ and $d'$ and
$d+d'=n$, we can define $i(\alpha,\beta)$ to be the intersection
number of $\alpha$ and $\beta$.  We can extend $i$ bilinearly to define $i(a,b)$ when $a$ and
$b$ are transverse smooth chains.  Since the intersection number is
invariant under small perturbations, we can also define $i(a,b)$ for
some non-transverse, non-smooth chains; it suffices that $a$ and $b$
are Lipschitz chains such that $\supp
\partial a \cap \supp b=\emptyset$ and $\supp a \cap \supp \partial
b=\emptyset$.

\begin{remark}\label{rem:genericIntersection}
  This intersection number is defined for ``generic'' $a$ and $b$.
  For instance, if $a$ and $b$ are Lipschitz chains of complementary
  dimension, then $\supp
  \partial(g\cdot  a) \cap \supp b=\emptyset$ and $\supp g\cdot a \cap \supp \partial
  b=\emptyset$ for all but a measure zero set of $g\in G$, so
  $i(g\cdot a,b)$ is defined for all but a measure zero set of $g\in G$.
\end{remark}

For every simplex $\sigma$ of $\tau$, let
$[\sigma]$ and $[\sigma^*]$ be the fundamental classes of $\sigma$ and
$\sigma^*$.  Then $\phi_\sharp([\sigma])$ and
$\phi_\sharp([\sigma^*])$ are Lipschitz chains in $G$ and we can
require that they be oriented so that their intersection number is 1.  By abuse of
notation, we will identify chains and cycles in $\tau$ with their
images in $G$ when it is clear, and simply write
$i([\sigma],[\sigma^*])$.

Let $a$ be a Lipschitz $d$-chain in $G$.  If $i(a,[\sigma^*])$
is defined for all cells $\sigma$ in $\tau^{(d)}$, define
$$P_{\tau}(a):=\sum_{\sigma\in \tau^{(d)}}
i(a,[\sigma^*]) \sigma \in C_{d+1}(\tau),$$ where $\tau^{(d)}$
is the $d$-skeleton of $\tau$, considered as the set of $d$-cells.
This is a simplicial $d$-chain which approximates $a$, and standard
arguments show that
$$\partial P_{\tau}(a)=P_{\tau}(\partial a).$$
Furthermore, if $a$ is a simplicial chain, then $P_{\tau}(a)=a$.

More generally, if $i(a,f_{\sharp}([\sigma^*]))$ is defined for all
cells $\sigma$ in $\tau$, define
$$P_{f(\tau)}(a):=\sum_{\sigma\in \tau^{(d)}}
i(a,f_\sharp([\sigma^*])) f_\sharp(\sigma)\in C^{\text{sm}}_{d}(G).$$
As above, 
$$\partial P_{f(\tau)}(a)=P_{f(\tau)}(\partial a),$$
so if $a$ is a cycle, so is $P_{f(\tau)}(a)$.  Define
$$P_{g\cdot f(\tau)}(a):=P_{(\rho_g\circ f)(\tau)}$$
where $\rho_g:X\to X$ is the map $\rho_g(x)=gx$.
Note that if $P_{f(\tau)}(g^{-1}\cdot a)$ is
defined, then 
$$P_{g\cdot f(\tau)}(a)=g\cdot P_{f(\tau)}(g^{-1}\cdot a).$$

We claim that this satisfies the properties in Prop.\
\ref{prop:PandQ}.  We have already seen that it is defined for all but
a measure-zero set of $g\in G$, and it is clear from the
definition that it is an element of $g\cdot f_\sharp
(C_d(\tau))$.
It remains to bound its mass.

Our arguments are based on the following lemma:
\begin{lemma}\label{lem:intersect}
  There is a $c_\cap>0$ depending on $G$ such that if
  $m+n=\dim G$, $a$ is an $m$-dimensional smooth chain, and $b$ is an
  $n$-dimensional smooth chain, then
  $$\int_G |i(g\cdot a, b)|\; dg \le c_\cap \mass a\mass b.$$
\end{lemma}
\begin{proof}
  We can consider the case that $a$ and $b$ both consist of single
  smooth simplices of diameter $\le 1$.  Any other chain can be
  barycentrically subdivided until all of its simplices have diameter
  $\le 1$, so the general case follows by linearity.  

  First, because $G$ is unimodular, we have
  $$  \int_G |i(g\cdot a, b)|\; dg = \int_G |i(g h\cdot a, h'\cdot b)|\; dg$$
  for any $h,h'\in G$, so, after translating $a$ and $b$, we may
  assume that $a$ and $b$ are contained in a ball of radius $1$ around
  $0\in G$.

  Let $\Delta^m$ and $\Delta^n$ be the standard euclidean simplices of
  dimension $m$ and $n$, and let $\alpha:\Delta^m\to G$ and
  $\beta:\Delta^n\to G$ be the maps corresponding to $a$ and $b$
  respectively.  Let $\gamma:\Delta^m\times \Delta^n\to G$ be the map
  $(x,y)\mapsto \beta(y)\alpha(x)^{-1}$.  If $g\cdot \alpha$ and $\beta$ are
  transverse, then $|i(g\cdot a,b)|\le
  \#\gamma^{-1}(g)$, and by the coarea formula, we have
  $$\int_G \#\gamma^{-1}(g)\; dg\le \vol \gamma.$$
  We can write $\gamma$ as the
  composition of maps $\gamma_0:\Delta^m\times \Delta^n\to G\times
  G$ and $p:G\times  G\to G$, where $\gamma_0(x,y)=(\alpha(x),\beta(y))$ and $p(h,k)=kh^{-1}$.  The image of $\gamma_0$ is contained in a ball of radius $2$; if we
  denote this ball by $B$, we have
  $$\vol \gamma\le (\vol \gamma_0)(\Lip p|_{B})^{\dim G}=(\vol
  \alpha)(\vol \beta)(\Lip p|_{B})^{\dim G},$$ where $\Lip p|_{B}$ is
  the Lipschitz constant of $p|_{B}$.  Thus the lemma holds for
  $c_\cap:=(\Lip p|_{B})^{\dim G}$.  
\end{proof}

Equation \eqref{eq:Pbound} follows easily:
\begin{lemma}\label{lem:volBound}
  Let $Z\subset G$ be a compact fundamental domain for the right
  action of $\Gamma$ on $G$.  Let $d\in \Z$. If $c_\cap$ is as in
  Lemma~\ref{lem:intersect}, then for all smooth $d$-cycles $\alpha$,
  $$\int_Z \|P_{g\cdot f(\tau)}(\alpha)\|_1 \; dg \le c_\cap \mass{\alpha} \sum_{\Gamma\cdot{\sigma}\in  \Gamma \backslash \tau^{(d)}}
  \mass f_\sharp([\sigma^*]).$$
  Consequently, there is an $g\in G$ such that
  $$\|P_{g\cdot f(\tau)}(\alpha)\|_1 \le c_\cap \frac{\mass{\alpha}}{\vol
    Z}\sum_{\Gamma\cdot{\sigma} \in \Gamma
    \backslash \tau^{(d)}}   \mass f_\sharp([\sigma^*]).$$
\end{lemma}
\begin{proof}
  Using Lemma~\ref{lem:intersect}, we see that
  \begin{align*}
    \int_Z \|P_{g\cdot f(\tau)}(\alpha)\|_1 \; dg & = \int_Z \sum_{\sigma\in
      \tau^{(d)}} |i(\alpha,g\cdot f_\sharp([\sigma^*]))| \; dg\\
    & \le \sum_{\Gamma\cdot{\sigma}\in  \Gamma \backslash \tau^{(d)}}
    \int_G |i(\alpha,g\cdot f_\sharp([\sigma^*]))| \; dg \\
    & \le  c_\cap \mass{\alpha} \sum_{\Gamma\cdot{\sigma}\in  \Gamma \backslash \tau^{(d)}}
    \mass f_\sharp([\sigma^*]),
  \end{align*}
  as desired.
\end{proof}

To finish the proof of Prop.\ \ref{prop:PandQ}, we need to construct
$Q_{g\cdot f(\tau)}(\alpha)$.  
\begin{lemma}\label{lemma:Q}
  For every smooth $d$-cycle $\alpha$ and for almost every $g\in G$, 
  there is a smooth $(d+1)$-chain $Q_{g\cdot f(\tau)}(\alpha)$ such that
  $$\partial Q_{g\cdot f(\tau)}(\alpha)=\alpha-P_{g\cdot f(\tau)}(\alpha).$$

  If $Z$ is a fundamental domain for the right action of $\Gamma,$
  then there is a $c_Q>0$ independent of $\alpha$ such that for all
  simplicial $d$-cycles $\alpha$,
  $$\int_{Z} \mass Q_{g\cdot f(\tau)}(\alpha) \; dg\le  c_Q \mass \alpha .$$
\end{lemma}

Any Lipschitz cycle can be approximated by a simplicial cycle, so the
requirement that $\alpha$ is simplicial is only a minor limitation.
The main idea of the proof is that if $\sigma$ is a simplex in
$\alpha$, then $P_{g\cdot f(\tau)}(\sigma)$ is only a bounded distance
from $\sigma$.  This lets us construct a chain homotopy between
$P_{g\cdot f(\tau)}$ and the identity.

\begin{proof}
  Since $G$ is a 1-connected nilpotent Lie group, it is 
  contractible.  One contraction is given by the scaling
  automorphisms; let $h:G\times [0,1]\to G$ be the map
  $h(g,t)= s_t(g)$.  This is smooth, and if $\gamma$ is a
  smooth $k$-cycle in $G$, then $h_\sharp(\gamma\times [0,1])$ is a
  smooth $(k+1)$-chain which fills $\gamma$.  Furthermore, there is
  a $w:\R\to \R$ such that if $\gamma$ is supported in the ball of
  radius $r$ around $0$, then $h_\sharp(\gamma\times [0,1])$ is
  supported in the same ball and
  $$\mass h_\sharp(\gamma\times [0,1]) \le w(r) \mass \gamma.$$

  If $\gamma$ is a non-zero smooth cycle in $G$, let $g_\gamma\in
  G$ be a point in the support of $\gamma$.  Define
  $h':C^{\text{lip}}_*(G)\to C^{\text{lip}}_{*+1}(G)$ by letting
  $h'(0)=0$ and
  $$h'(\gamma)=g_\gamma\cdot h_\sharp(g_\gamma^{-1}\gamma\times
  [0,1])$$ for all $\gamma\ne 0$.  If $\gamma$ is a cycle and $\dim \gamma\ge 1$, then
  $\partial h'(\gamma)=\gamma$ and 
  $$\mass h'(\gamma) \le w(\diam \gamma)\mass \gamma,$$ where $\diam \gamma$ is the diameter of the
  support of $\gamma$.

  Let $C_*\subset C^{\text{sm}}_*(G)$ be the chain complex generated
  by smooth simplices which are transverse to the barycentric
  subdivision of $f$ (i.e., $f$ considered as a map $B(\tau)\to G$.)
  We will use $h'$ to construct a chain homotopy
  $Q_{f(\tau)}:C_*(G)\to C^{\text{sm}}_{*+1}(G)$ between $P_{f(\tau)}$
  and $\id$; that is, a linear map which satisfies the identity
  \begin{equation}\label{eq:chainHtpy}
    \partial Q_{f(\tau)}(\gamma)+Q_{f(\tau)}(\partial
    \gamma)=P_{f(\tau)}(\gamma)-\gamma
  \end{equation}
  for all $\gamma\in C_*$.  In particular, if $\gamma$ is a
  cycle, then $\partial
  Q_{{f(\tau)}}(\gamma)=P_{f(\tau)}(\gamma)-\gamma$.

  We define $Q_{{f(\tau)}}$ inductively.  The base case is to define $Q_{{f(\tau)}}$ on
  0-simplices; if $g\in G$, let $Q_{{f(\tau)}}(g)=h'(P_{f(\tau)}(g)-g)$; this
  satisfies \eqref{eq:chainHtpy}.  In general,
  if we have defined $Q_{{f(\tau)}}$ on simplices of dimension at most $k$ so
  that it satisfies \eqref{eq:chainHtpy} and $s$ is a
  $(k+1)$-simplex, then
  $$P_{f(\tau)}(s)-s-Q_{{f(\tau)}}(\partial s)$$
  is a $k$-cycle.  We define $Q_{{f(\tau)}}$ on $s$ by letting
  $$Q_{{f(\tau)}}(s)=h'(P_{f(\tau)}(s)-s-Q_{{f(\tau)}}(\partial s)).$$

  Define
  $$Q_{g\cdot f(\tau)}(a)=g\cdot Q_{f(\tau)}(g^{-1}\cdot a).$$
  If $\gamma$ is a cycle, then
  $$\partial  Q_{g\cdot f(\tau)}(\gamma)=P_{g\cdot f(\tau)}(\gamma)-\gamma.$$

  Let $D=\max_{\sigma\in \tau}\diam f(\sigma)$ be the maximum diameter
  of the image of a simplex of $\tau$ and let
  $$D^*=\max_{\sigma\in \tau} \diam f(\sigma^*)$$ 
  be the maximal diameter of a dual cell.  If $S\subset X$, let $N_r(S)$ be the
  neighborhood of $S$ of radius $r$.  We will show that for all $0\le
  i\le \dim G$ there are numbers $M_i, D_i>0$ such that if $\sigma$ is a
  smooth $i$-simplex whose diameter is at most $D$,
  then
  $$\supp Q_{f(\tau)}(\sigma)\subset N_{D_i}(\supp
  \sigma)$$
  if $\sigma$ is transverse to $f$ 
  and
  $$\int_{Z} \mass Q_{g\cdot f(\tau)}(\sigma) \; dx=\int_{Z} \mass
  Q_{f(\tau)}(g^{-1}\cdot \sigma) \; dx\le  M_i(\mass
  \sigma+\mass \partial\sigma).$$ 

  We proceed by induction.  First, consider the case that $\dim
  \sigma=0$.  In this case, $\sigma$ is supported on a point $x$, and
  if $\sigma$ is transverse to $f$, then the support of
  $P_{f(\tau)}(\sigma)$ is contained in a ball around $x$ of radius
  $D$.  Therefore $\diam P_{f(\tau)}(\sigma)-\sigma\le 2D$ and
  $\diam Q_{f(\tau)}(\sigma)\le 4D$.  For all but a measure zero
  subset of $G$, we have
  $$\mass Q_{f(\tau)}(g^{-1}\cdot\sigma)\le w(2D)\mass
  P_{f(\tau)}(g^{-1}\cdot\sigma).$$ By Lemma~\ref{lem:volBound}, there
  is a $M_0$ such that
  $$\int_Z \mass Q_{f(\tau)}(g^{-1}\cdot\sigma)\;dg \le M_0.$$

  Assume by induction that the mass and diameter bounds hold for
  simplices of dimension $< k$ and consider the case that $\sigma$ is
  a $k$-simplex  transverse to $f$ .  Then $Q_{f(\tau)}(\sigma)=h'(\rho(\sigma)),$ where
  $$\rho(\sigma)=P_{f(\tau)}(\sigma)-\sigma-Q_{{f(\tau)}}(\partial\sigma),$$
  so we consider $\supp P_{f(\tau)}(\sigma)$ and
  $\supp Q_{{f(\tau)}}(\partial\sigma)$.  By the definition of
  $D$, we have 
  $$\supp P_{f(\tau)}(\sigma)\subset N_D(\supp \sigma).$$  Each
  simplex of $\partial \sigma$ has diameter at most $D$, so by
  induction,
  $$\supp Q_{{f(\tau)}}(\partial\sigma)\subset N_{D_{k-1}}(\supp
  \sigma).$$
  The diameter of $\rho(\sigma)$ is thus at most $2 D_{k-1}+3 D$, so
  the
  diameter of $Q_{{f(\tau)}}(\sigma)=h'(\rho)$ is at most $D_k:=4 D_{k-1}+6
  D$, and 
  $$\supp Q_{{f(\tau)}}(\sigma)\subset N_{D_{k}}(\supp \sigma).$$

  Furthermore, we have
  $$\mass Q_{f(\tau)}(\sigma) \le w(D_k) \mass \rho(\sigma).$$
  By Lemma~\ref{lem:volBound} and the inductive hypothesis, there
  is a $c'$ such that
  $$\int_Z \mass \rho(g^{-1}\cdot\sigma)\;dg \le c'(\mass
  \sigma+\mass \partial\sigma),$$
  so if $M_k:=w(D_k)c'$, then 
  $$\int_Z \mass Q_{f(\tau)}(g^{-1}\cdot\sigma)\;dg \le M_k(\mass
  \sigma+\mass \partial\sigma)$$
  as desired.
\end{proof}
This completes the proof of Prop.~\ref{prop:PandQ}.

We can also construct chains which interpolate between two
different approximations of a cycle.  For example, if $\tau_1$ and $\tau_2$
are different smooth triangulations of $G$, we can construct a chain
interpolating between $P_{\tau_1}(\alpha)$ and $P_{\tau_2}(\alpha)$.
Let $\tau'$ be a smooth triangulation of $G\times [0,1]$ which
restricts to $\tau_1$ on $G\times \{0\}$ and to $\tau_2$ on $G\times
\{1\}$.  Consider the $(d+1)$-chain
$$H=P_{\tau'}(\alpha\times [0,1]).$$
Its boundary is
$$\partial H=P_{\tau'}(\alpha\times \{0\}-\alpha \times
\{1\})=P_{\tau_1}(\alpha)\times \{0\}-P_{\tau_2}(\alpha)\times
\{1\},$$ so the projection of $H$ to $G$ has boundary
$P_{\tau_1}(\alpha)-P_{\tau_2}(\alpha)$.  

Similarly, if $f_1:\tau\to G$ and $f_2:\tau\to G$ are homotopic by a homotopy
$h:\tau\times [0,1]\to G$, we can construct a $(d+1)$-chain connecting
$P_{f_1(\tau)}(\alpha)$ and $P_{f_2(\tau)}(\alpha)$ in a similar way.
We define
$$ H'=P_{h(\tau\times [0,1])}(\alpha\times [0,1])$$
and consider the projection of $H'$ to $G$; this has boundary 
$P_{f_1(\tau)}(\alpha)-P_{f_2(\tau)}(\alpha)$ as desired.

To bound the volume of constructions like this, we need a corollary of
Lemma~\ref{lem:intersect}:
\begin{cor}\label{cor:intervalIntersect}
  If $m+n=\dim G$, $a$ is an $m$-dimensional smooth chain
  in $G$, and $b$ is an $n$-dimensional smooth chain in
  $G\times [0,1]$, then $i(g\cdot a\times [0,1],b)$ is defined for all
  $g\in G$ except for a measure 0 subset, and
  $$\int_G |i(g\cdot a\times [0,1],b)|\;dg\le c_\cap \mass a \mass b.$$
\end{cor}
\begin{proof}
  As before, we can reduce to the case that $a$ and $b$ are single smooth
  simplices, given by maps $\alpha$ and $\beta$.  Let
  $p:G\times[0,1]\to G$ be the projection map and consider $p \circ
  \beta$.  Each transverse intersection between $g\cdot \alpha$ and
  $p\circ \beta$ corresponds to a transverse intersection between
  $g\cdot \alpha \times [0,1]$ and $\beta$, and vice versa, so
  $$\int_G |i(g\cdot \alpha\times [0,1],\beta)|\;dg =\int_G |i(g\cdot
    \alpha,p\circ \beta)|\;dg\le c_\cap \vol \alpha \vol \beta$$
\end{proof}
We can then bound the mass of $H$ and $H'$:
\begin{lemma}\label{lem:homBound}
  Let $\tau'$ be a $\Gamma$-adapted smooth triangulation of $G\times
  [0,1]$, and let $h:\tau'\to G\times [0,1]$ be a $\Gamma$-equivariant
  map which is piecewise smooth.  Let $d\in \Z$ and
  let $Z$ be a compact fundamental domain for the right action of $\Gamma$
  on $G$.  For all smooth $d$-cycles $\alpha$ in $G$,
  $$\int_Z \|P_{g\cdot h(\tau')}(\alpha \times [0,1])\|_1 \; dg \le c_\cap
  \mass{\alpha} \sum_{\Gamma\cdot{\sigma}\in  \Gamma \backslash (\tau')^{(d+1)}}
  \mass h_\sharp([\sigma^*]).$$
  Consequently, there is an $g\in G$ such that
  $$\|P_{g\cdot h(\tau')}(\alpha \times [0,1])\|_1 \le c_\cap \frac{\mass{\alpha}}{\vol_n Z} \sum_{\Gamma\cdot{\sigma} \in \Gamma
    \backslash (\tau')^{(d+1)}} \mass h_\sharp([\sigma^*]).$$
\end{lemma}
\begin{proof}
  Corollary~\ref{cor:intervalIntersect}
  implies that 
  \begin{align*}
    \int_Z \|P_{g\cdot h(\tau')}(\alpha\times [0,1])\|_1 \; dg &  \le \sum_{\Gamma\cdot{\sigma}\in  \Gamma \backslash (\tau')^{(d+1)}}
    \int_G |i(\alpha \times [0,1],g\cdot h_\sharp([\sigma^*]))| \; dg\\
    & \le  c_\cap\mass{\alpha} \sum_{\Gamma\cdot{\sigma}\in  \Gamma \backslash (\tau')^{(d+1)}}
    \mass h_\sharp([\sigma^*]),
  \end{align*}
  as desired.
\end{proof}

\section{Filling cycles in Carnot groups}\label{sec:mainthm}
Now we apply the approximation techniques of the previous
section to construct fillings of cycles in Carnot groups.  Let $G$ be
an $n$-dimensional Carnot group, let $s_t:G\to G, t>0$ be the family
of scaling automorphisms, and let $\Gamma$ be a lattice in $G$ such
that $s_{2^i}(\Gamma)\subset \Gamma$ for $i=1,2,\dots$.  

If $\alpha$ is a Lipschitz cycle and $f:\tau \to G$, we can construct
a sequence
$$P_{(s_{2^0}\circ f)(\tau)}(\alpha),\dots, P_{(s_{2^i}\circ f)(\tau)}(\alpha),\dots$$
of approximations of $\alpha$ of different scales.  Furthermore, using
the constructions in the previous section, we can construct chains
interpolating between these approximations.  If we can bound the mass of the
interpolating chains and show that $P_{(s_{2^i}\circ f)(\tau)}(\alpha)=0$ for some
$i$, we can bound the filling volume of $\alpha$.

In general, we cannot expect good bounds on the mass of
$P_{(s_{2^i}\circ f)(\tau)}(\alpha)$, but we can bound the expected
mass of $P_{g\cdot (s_{2^i}\circ f)(\tau)}(\alpha)$ when $g$ is a
random element of $G$.  The masses of the approximations and the
interpolating chains depend on the masses of simplices of $s_{2^i}\circ
f$, so this lets us find bounds on $\FV^*_G$ by constructing maps from
simplicial complexes to $G$ that satisfy certain metric properties.

We first state some definitions necessary to state the main
proposition of the section.  If $(\tau,\phi:\tau\to G)$ is a
triangulation, we define $s_t(\tau)$ to be the triangulation
$(\tau,s_t\circ \phi)$.  If $\tau$ is $\Gamma$-adapted, then
$s_t(\tau)$ is $s_t(\Gamma)$-adapted.  

If $S\subset G$ is an open
subset, then the scaling automorphisms polynomially expand the volume
of $S$.  That is, there is an integer $\kappa$ such that
$$\vol s_t(S)=t^{\kappa} \vol S.$$
We call $\kappa$ the {\em volume growth exponent} of $G$.  
Likewise, if $f:\tau\to G$, the scaling automorphisms polynomially
expand the volume of $f(\tau^{(d)}$ for all $d$.  If
$$\size_d(f)=\sum_{\Gamma\cdot{\sigma}\in  \Gamma \backslash B(\tau)^{(d)}} \mass
  f_\sharp(\sigma^*),$$
there are $k(f,d)$ such that
$$\size_{d}(s_t\circ f) \sim t^{k(f,d)}$$
for $t>1$.


\begin{prop}\label{prop:mainProp}
  Let $\tau=(\tau,\phi)$ be a $\Gamma$-adapted triangulation of $G$
  and let $\tau'$ be an $s_2(\Gamma)$-adapted triangulation of
  $G\times [1,2]$ which restricts to $\tau$ on $G\times\{1\}$ and
  $s_2(\tau)$ on $G\times\{2\}$.
  Let $f:\tau\to G$ be a $s_2(\Gamma)$-equivariant piecewise smooth
  map $f:G\times [1,2]\cong \tau' \to G$ which is $\Gamma$-equivariant
  when restricted to $G\times \{1\}$ and satisfies
  $f(g,2)=s_2(f(s_{1/2}(g),1))$.

  Let $n=\dim G$, and for each $1\le d\le n$, let $k(d)=k(f,d)$.  Let
  $\kappa$ be the volume growth exponent of $G$.  Then for
  all $1\le d < n$, if $k(d+1)+k(n-d)>\kappa$, then
  $$\FV^d_G(V)\precsim  V^\frac{k(d+1)}{\kappa-k(n-d)}.$$
\end{prop}
\begin{proof} 
  It suffices to construct fillings for all simplicial $d$-cycles.  
  Let $\alpha\in C_d(\tau)$ be a simplicial $d$-cycle with mass $V$.
  Let $c_0$ be such that for all $t>1$ and all $1\le d'\le \dim G$,
  $$c_\cap \size_{d'}(s_t\circ f,\tau') \le c_0 t^{k(d')}.$$
  Let $\izero$ be the smallest positive integer such that
  $$2 c_0 V 2^{\izero k(n-d)} < 2^{\kappa \izero} \vol Z;$$ 
  note that $k(d')<\kappa$ for all $d'<n$, so this exists.
  We will construct a filling of $\alpha$ by constructing a
  triangulation $\tau_0$ of $G\times [1,2^{\izero}]$ and a map
  $f_0:G\times [1,2^\izero]\to G\times [1,2^\izero]$, then considering
  $$P_{f_0(\tau_0)}(\alpha\times [1,2^{\izero}]).$$

  We build $\tau_0$ out of scalings of $\tau'$.  The conditions
  on $\tau'$ mean that scaled copies of $\tau'$ can be glued together.
  That is, we extend $s_t$ to a map $s_t:G\times [0,\infty) \to
  G\times [0,\infty)$ given by $s_t(g,x)=(s_t(g), xt)$ and define
  $\tau'_i:=s_{2^{i-1}}(\tau')$.  For each $i$, $\tau'_i$ is a
  triangulation of $G\times [2^{i-1},2^{i}]$, and $\tau'_i$ restricts
  to $s_{2^{i-1}}(\tau)$ on $G\times \{2^{i-1}\}$ and to
  $s_{2^{i}}(\tau)$ on $G\times \{2^{i}\}$.  In particular, $\tau'_i$
  and $\tau'_{i+1}$ agree on $G\times \{2^{i}\}$, and we can glue
  $\tau'_1,\dots, \tau'_{\izero}$ to obtain a triangulation $\tau_0$
  of $G\times [1,2^{\izero}]$.

  The conditions on $f$ let us extend it to all of $\tau_0$.
  We define $f_0:G\times [1,2^{\izero}]\to G \times
  [1,2^{\izero}]$ by letting
  $$f_0(p)=s_{2^i}\circ f \circ s_{2^{-i}}(p)$$
  for all $p\in G\times [2^i,2^{i+1}]$.  Since $s_{2}\circ f \circ
  s_{2^{-1}}(p)=f(p)$ for all $p\in G\times \{2\}$, this is
  well-defined and continuous.

  Now, for $x\in G$, consider the chain
  $$\beta_0(x):=P_{x\cdot f_0(\tau_0)}(\alpha\times [1,2^{\izero}]).$$
  Let $\bar{f}:\tau \to G$ be the map $\bar{f}(y)=f(y,1)$.
  The boundary of $\beta_0(x)$ is then
  $$\partial\beta_0(x)=P_{x\cdot \bar{f}(\tau)}(\alpha) \times \{1\} -
  P_{x\cdot (s_{2^{\izero}}\circ \bar{f})(\tau)}(\alpha) \times \{2^{\izero}\}.$$
  This boundary has two pieces: an approximation of $\alpha$ in
  $x\cdot \bar{f}(\tau)$ and an approximation of $\alpha$ in $x\cdot (s_{2^{\izero}}\circ \bar{f})(\tau)$.  By Lemma~\ref{lemma:Q}, the first piece
  is usually close to $\alpha$.  We will show that the second term is typically zero and
  that there is a choice of $x$ which leads to a $\beta_0$ with low mass.

  First, we claim that the ``large-scale'' term 
  $$\gamma(x):=P_{x\cdot s_{2^{\izero}}(f(\tau))}(\alpha)$$
  of $\partial\beta_0(x)$ usually vanishes.  
  Let $Z$ be a fundamental domain for the right action of $\Gamma$ on
  $G$.
  By Lemma~\ref{lem:volBound},
  $$\int_{s_{2^{\izero}}(Z)} \|\gamma(x)\|_1 \; dx \le  c_0
  2^{\izero k(n-d)} V < \frac{2^{\kappa\izero}\vol Z}{2} = \frac{\vol s_{2^{\izero}}(Z)}{2},$$
  so there is a set $Z'\subset s_{2^{\izero}}(Z) $ of volume $\vol Z' > \vol s_{2^{\izero}}(Z)/2$ such that if
  $x\in Z'$, then 
  $$\|\gamma(x)\|_1 < 1.$$
  Since $\gamma(x)$ is an integral chain, this implies that if
  $x\in Z'$, then $\gamma(x)=0$.  Let $p:G\times [1,2^\izero]\to G$ be
  the projection to the first coordinate.  If we let
  $$\beta(x):=Q_{x\cdot f(\tau)}(\alpha)+p_\sharp(\beta_0(x)),$$
  then $\partial\beta(x)=\alpha$ when $x\in Z'$.

  We claim that an appropriate choice of $x\in Z'$ leads to a
  $\beta(x)$ with small mass.  Let $Z_i:=s_{2^i}(Z)$ and let
  $f_i:G\times [2^{i-1},2^i]\to G$ be the restriction of $f_0$ to
  $G\times [2^{i-1},2^i]$.  Recall that
  $s_{2^i}(\Gamma)\subset s_{2^j}(\Gamma)$ whenever $i>j$, so if
  $\lambda:G\to \R$ is $s_{2^j}(\Gamma)$-invariant, then 
  $$\int_{Z_i}\lambda(x)\;dx=[s_{2^i}(\Gamma):s_{2^j}(\Gamma)]\int_{Z_j}\lambda(x)\;dx=2^{(i-j)\kappa}\int_{Z_j}\lambda(x)\;dx.$$
  Consider
  $$\int_{Z_\izero} \mass{\beta(x)} \; dx.$$
  We can break $\beta(x)$ into a sum
  $$\beta(x)=Q_{x\cdot f(\tau)}(\alpha)+\sum_{i=1}^{\izero} p_\sharp(P_{x\cdot
    f_i(\tau'_i)}(\alpha\times[2^{i-1},2^{i}]))$$ and bound the
  average mass of each term.  First, 
  \begin{align*}
    \int_{Z_\izero} \mass{Q_{x\cdot f(\tau)}(\alpha)} \; dx
    &=2^{\kappa \izero}\int_{Z} \mass{Q_{x\cdot
        f(\tau)}(\alpha)}\; dx \\
    &\le 2^{\kappa \izero} c_Q'V.
  \end{align*}
  Next, for $1\le i\le \izero$,
  \begin{align*}
    \int_{Z_\izero} 
    &\mass{p_\sharp(P_{x\cdot f_i(\tau'_i)}(\alpha\times[2^{i-1},2^{i}]))}\; dx \\
    &=
    2^{(\izero-i)\kappa } \int_{Z_i}
    \mass{p_\sharp(P_{x\cdot f_i(\tau'_i)}(\alpha\times[2^{i-1},2^{i}]))}\; dx \\
    &\le 2^{(\izero-i)\kappa } \int_{Z_i}
    c_0 2^{(i-1) k(d+1)}
    \|P_{x\cdot
      f_i(\tau'_i)}(\alpha \times [2^{i-1},2^{i}])\|_1 \; dx \\
    & \le 2^{(\izero-i)\kappa } c_0 2^{(i-1) k(d+1)}
    c_0 2^{(i-1) k(n-d)}V
  \end{align*}
  Thus
  $$\mass\beta(x)\le  [2^{\kappa \izero} c_Q'+\sum_{i=1}^{\izero}c_0^2
  2^{(\izero-i) \kappa +(i-1)[k(d+1)+k(n-d)]}]V.$$
  Since $k(d+1)+k(n-d)>\kappa$, the sum is dominated by the $i=\izero$
  term, so there is a $C$ such that
  $$\int_{Z_\izero} \mass{\beta(x)} \; dx \le C 2^{(k(d+1)+k(n-d))\izero}V$$
  Consequently, there is an $x\in Z'$ such that 
  $$\mass{\beta(x)} \le \frac{2 C}{\vol{Z_\izero}}
  2^{(k(d+1)+k(n-d)-\kappa)\izero}V.$$ Fix such an $x$.  If $V$ is
  sufficiently large, then the
  definition of $\izero$ implies that
  $$\biggl(\frac{2 c_0 V}{\vol Z}\biggr)^{\frac{1}{\kappa-k(n-d)}} < 2^\izero \le 2
  \biggl(\frac{2 c_0 V}{\vol Z}\biggr)^{\frac{1}{\kappa-k(n-d)}},$$
  so there is a $C'$ such that
  $$\mass{\beta(x)} \le C'
  V^{1+\frac{k(d+1)+k(n-d)-\kappa}{\kappa-k(n-d)}}=C'
  V^{\frac{k(d+1)}{\kappa-k(n-d)}},$$
  as desired.
\end{proof}

It remains to show that triangulations and maps with the desired
properties exist when $G$ is one of the Heisenberg groups.  The
Heisenberg groups are closely connected to contact geometry.  The
distribution of horizontal planes in the Heisenberg group forms a
contact structure, so results on the flexibility of isotropic maps
(see \cite{GroPDR} or \cite{ElMi}), imply that $H_n$ has many
horizontal submanifolds of dimension $\le n$.  In particular, one can
show (see for instance \cite[Lem.\ 12]{YoungFING}) that
\begin{lemma}\label{lem:heisTri}
  If $H_n$ is the $(2n+1)$-dimensional Heisenberg group, $\Gamma$ is the
  $(2n+1)$-dimensional integral Heisenberg group, and $s_t:H_n\to H_n,
  t>0$ is the family of scaling automorphisms of $H_n$, then there is a
  $s_2(\Gamma)$-adapted triangulation $\tau'$ of $H_n\times
  [1,2]$ and a $s_2(\Gamma)$-equivariant piecewise smooth map $f:H_n\times [1,2]\cong \tau'\to H_n$
  such that
  \begin{itemize}
  \item $\tau'$ restricts to a triangulation $\tau$ on $H_n\times \{0\}$
    and $s_2(\tau)$ on $H_n\times \{1\}$.
  \item For all $g\in H_n$, we have $f(s_2(g),2)=s_2(f(g,1))$.
  \item If $\sigma$ is a $d$-simplex of $\tau'$ for some $d\le n$, then
    $f|_\sigma$ is horizontal.
  \end{itemize}
\end{lemma}
In Lemma~12 of \cite{YoungFING}, we constructed a triangulation
$\tau_0$ of $H_n\times [1,2]$ and a map $f_0:\tau_0\to H_n$ which satisfy
all the conditions above except for the smoothness condition on $f_0$;
we only required that $f_0$ be Lipschitz.  In fact, we constructed
$f_0$ to be smooth on simplices of $B(\tau_0)$; we can thus take
$f=f_0$ and $\tau'=B(\tau_0)$.

This lets us prove Theorem~\ref{thm:DehnThm}:
\begin{proof}[{Proof of Theorem~\ref{thm:DehnThm}}]
  Since $f$ is horizontal on all simplices of $B(\tau')$ below dimension
  $n$, we can let $k(d)=d$ when $d\le n$.  Otherwise, recall that we
  can write the
  Lie algebra $\mathfrak{h}$ of $H_n$ as
  $$\mathfrak{h}=\R^{2n}\otimes \R.$$
  The scaling automorphism $s_t$ acts by scaling by $t$ on the $\R^{2n}$
  component and by scaling by $t^2$ on the $\R$ component.  In
  particular, there is a $c$ such that if $\beta$ is a wedge of $d$
  vectors in $\mathfrak{h}$, then
  $$\|{s_t}_*(\beta)\|\le c t^{d+1} \|\beta\|,$$
  which is sharp, for instance, when $\beta$ is a nontrivial wedge of
  $d-1$ vectors of $\R^{2n}$ and a generator of $\R$.  Thus if $\sigma$
  is a $d$-dimensional simplex of $B(\tau)$ for $d>n$, then there is a
  $c$ such that
  $$\vol s_t(f(\sigma))\le c t^{d+1},$$
  and we can let  $k(d)=d+1$ when $d>n$.  Similarly, the volume growth
  exponent of $H_n$ is $\kappa=2n+2$.  

  Thus, when $d>n$, we have $k(d+1)+k(2n+1-d)>2n+2$, so
  Prop.~\ref{prop:mainProp} implies that
  $$\FV^{d+1}_{H_n}(V) \lesssim V^{\frac{d+2}{d+1}}.$$
\end{proof}
Burillo \cite{BurLow} proved the corresponding lower bound, so this
implies that
$$\FV^{d+1}_{H_n}(V) \sim V^{\frac{d+2}{d+1}}.$$

\def\cprime{$'$}

\end{document}